\documentclass[11pt]{amsart}
\usepackage{amsmath,amssymb,amsthm,mathtools,xypic}
\usepackage[shortlabels]{enumitem}
\usepackage[hmargin=35mm, vmargin=30mm]{geometry}
\usepackage{xy} 
\usepackage{mathrsfs} 
\usepackage{xspace}
\usepackage[pagebackref=true, colorlinks]{hyperref}

\title{Compressible subgroups and simplicity}

\author{Alejandra Garrido} 
\address{Facultad de Matem\'aticas, Universidad Complutense de Madrid, and ICMAT, Madrid, SPAIN}
\email{alejandra.garrido@ucm.es; alejandra.garrido@icmat.es}

\author{Colin D. Reid}
\address{School of Information and Physical Sciences, University of Newcastle Australia}
\email{colin@reidit.net}

\thanks{This research has been financially supported by the Australian Research Council (grant FL170100032) and by Spain’s Ministry of Science and Innovation (grants [PID2020-114032GB-I00] and [CEX2019-000904-S])}

\newtheorem{thm}{Theorem}[section]

\newtheorem{prop}[thm]{Proposition}
\newtheorem{lem}[thm]{Lemma}
\newtheorem{cor}[thm]{Corollary}
\newtheorem{que}{Question}

\theoremstyle{definition}
\newtheorem{defn}[thm]{Definition}

\newtheorem{rmk}[thm]{Remark}

\newcommand{\Zb}{\mathbb{Z}}
\newcommand{\Nb}{\mathbb{N}}
\newcommand{\Rb}{\mathbb{R}}

\newcommand{\mc}[1]{\mathcal{#1}}
\newcommand{\mf}[1]{\mathfrak{#1}}
\newcommand{\ms}[1]{\mathscr{#1}} 

\newcommand{\tdlc}{t.d.l.c.\@\xspace}

\newcommand{\Aut}{\mathrm{Aut}}

\newcommand{\Homeo}{\mathrm{Homeo}}

\newcommand{\TMon}{\overline{\mathrm{M}}}
\newcommand{\rist}{\mathrm{rist}}
\newcommand{\con}{\mathrm{con}}

\newcommand{\N}{\mathrm{N}}
\newcommand{\CC}{\mathrm{C}}
\newcommand{\Full}{\mathrm{F}}

\newcommand{\Al}{\mathrm{A}}
\newcommand{\Der}{\mathrm{D}}

\newcommand{\inv}{^{-1}}
\newcommand{\supp}{\mathrm{supp}}

\newcommand{\defbold}{\textbf}

\newcommand{\triv}{\{1\}}

\newcommand{\grp}[1]{\langle #1 \rangle}

\newcommand{\ngrp}[1]{\langle \langle #1 \rangle \rangle}
\newcommand{\ol}[1]{\overline{#1}}

\usepackage[usenames, dvipsnames]{color}

\begin{document}

\begin{abstract}
In this article we give sufficient conditions for a group to have simple derived subgroup; the argument is based on generalising properties observed for extremely proximal micro-supported actions on the Cantor space, and generalises previous results of Matui, Le Boudec and others in this direction.  We give a sufficient condition for a non-trivial normal subgroup (not assumed closed) of a locally compact group $G$ to be open, also based on the theory of micro-supported actions. 
This shows in particular that many of the class of robustly monolithic groups introduced by Caprace--Reid--Wesolek are simple-by-discrete.
\end{abstract}


\maketitle

\tableofcontents

\addtocontents{toc}{\protect\setcounter{tocdepth}{1}}

\section{Introduction}

\subsection{Background}

Let $G$ be a group acting faithfully on a set $X$; given $Y \subseteq X$ we write $Y^c := X \setminus Y$ and define the \defbold{rigid stabiliser} $\rist_G(Y)$ to be the fixator (that is, pointwise stabiliser) of $Y^c$.  There are a number of simplicity theorems in the literature based on the existence of a $G$-invariant family $\mc{C}$ of subsets with the following properties:
\begin{enumerate}[(a)]
\item For every $g \in G \setminus \{1\}$, there is $Y \in \mc{C}$ such that $gY \cap Y = \emptyset$.
\item The subgroups $\rist_G(Y)$ are non-trivial.
\item The elements of $\mc{C}$ can all be mapped inside each other by elements of $G$.
\end{enumerate}

Criterion (a) is automatic if, for example, $X$ is a Hausdorff space on which $G$ acts by homeomorphisms and $\mc{C}$ is a base of topology.  So the interesting content in terms of the action is usually in criteria (b) and (c).  The conclusion is usually that $G$ is monolithic or almost simple, in the following sense.

\begin{defn}
The \defbold{monolith} of a group $G$ is the intersection of all its non-trivial normal subgroups, and $G$ is \defbold{monolithic} if the monolith is non-trivial, that is, there is a unique smallest non-trivial normal subgroup.  A group $G$ is \defbold{almost simple} if the monolith of $G$ is non-abelian and simple.
\end{defn}

An illustrative example is Tits' Property (P) theorem.  We say an action on a tree is \defbold{geometrically dense} if the action preserves no end or proper subtree of the tree; this implies that every half-tree can be mapped inside every other half-tree (Lemma~\ref{lem:tree:geom_dense}), which is an instance of criterion (c).  Given a group $G$ of automorphisms of a tree, we write $G^+$ for the group generated by arc stabilisers and $G^{++}$ for the group generated by rigid stabilisers of half-trees in $G$.

\begin{thm}[{\cite[Th\'{e}or\`{e}me~4.5]{Tits70}}]\label{thm:Tits}
Let $T$ be a tree and let $G \le \Aut(T)$ act geometrically densely on $T$.  Suppose also that the fixator of each (finite or infinite) path $P$ is a direct product of elements supported on individual fibres of the closest point projection from $T$ to $P$, and that $G^+$ is non-trivial.  Then $G$ is almost simple with monolith $G^+ = G^{++}$.
\end{thm}

Other authors have subsequently proved more general simplicity criteria groups acting on trees (e.g. \cite{LB-AP}, \cite{MollerVonk}), as well as for groups acting on other combinatorial structures (e.g. \cite{HaglundPaulin}, \cite{Caprace-buildings}, \cite{Lazarovich}, \cite{deMedtsSilvaStruyve}).  Taken together, Tits' theorem and the results it has inspired form a major source of examples of non-linear simple groups, especially for the theory of totally disconnected, locally compact (\tdlc) groups.

A second motivating example, and the original motivation of the authors in writing this article, comes from the theory of topological full groups.  We first introduce a special case of criteria (a) and (c).

\begin{defn}
We take the convention that all (locally) compact spaces are required to be Hausdorff.  Let $X$ be a Hausdorff space and let $G$ be a group acting on $X$ by homeomorphisms.  A subset $Y \subseteq X$ is \defbold{$G$-compressible} if for all non-empty open $O \subseteq X$, there exists $g \in G$ such that $gY \subseteq O$.  We say the action is \defbold{compressible} if a non-empty $G$-compressible open subset exists and \defbold{fully compressible} (or \defbold{extremely proximal}, see for instance \cite[\S2.3]{LeBoudecURS} and articles referenced there) if every nondense open set is $G$-compressible.

Given a topological space $X$ and $G \le \Homeo(X)$, the \defbold{piecewise full group} (or \defbold{topological full group}) $\Full(G)$ consists of all homeomorphisms $\phi$ from $X$ to $X$, such that there exists an open cover $\mc{O}$ of $X$ and elements $g_O \in G$ such that $\phi(x) = g_O(x)$ for all $x \in O \in \mc{O}$.

Throughout, we write $\Der(G)$ for the \defbold{derived subgroup} of $G$.
\end{defn}

The general properties of piecewise full groups are easiest to analyse when $X$ is zero-dimensional.
In the case that $X$ is the Cantor space and the action is minimal, Nekrashevych \cite[Theorem~4.1]{Nekra} showed that the group $\Full(G)$ is almost simple; the monolith $\Al(G)$ has an explicit generating set, but the structure of the quotient $\Full(G)/\Al(G)$ is mysterious in general.  
More directly relevant to the present article is the following theorem of Matui (paraphrased):

\begin{thm}[{See \cite[Proposition~4.11]{MatuiSimple} and \cite[Theorem~4.16]{MatuiSimple}}]\label{thm:Matui}
Let $X$ be the Cantor space and let $G \le \Homeo(X)$ be minimal and fully compressible.  Then $\Full(G)$ is almost simple with monolith $\Der(\Full(G))$.
\end{thm}

In other words, in this case the quotient $\Full(G)/\Al(G)$ is abelian.  It is also easy to see that $\Full(G)$ is generated by elements of $G$-compressible support, so criteria (a)--(c) are satisfied.  Note, however, that in \cite[Theorem 4.1]{Nekra} the full strength of (c) cannot be assumed because one of the main goals of the work \cite{Nekra} is to show that there are groups $G$ for which  $\Al(G)$ is amenable and it is known that minimal fully compressible actions are never amenable.

\subsection{Compression families of subgroups}

In this article we develop some general commutator arguments for proving that a group is almost simple.  To start, let us generalise the criteria (a)--(c) above to properties of a certain family of subgroups that will occur in a fully compressible action.

\begin{defn}\label{def:msc:subgroups}
	Let $G$ be a non-trivial group.  Let $\mc{C}$ be a family of subgroups of $G$ that is invariant under conjugation in $G$ and let $H \le G$.  Consider the following conditions:
	\begin{enumerate}[label=(\Alph*)]
		\item Every $A \in \mc{C}$ is non-abelian. \label{item:nonabelian}
		\item For every $g \in H \setminus \{1\}$, there is $A \in \mc{C}$ such that $[A,gAg\inv]=\triv$. \label{item:disjointtranslates}
		\item For all $A,B \in \mc{C}$ there is $g \in H$ such that $gAg\inv \le B$. \label{item:semitransitive}
		\item Given $A,B,C \in \mc{C}$ such that $[A,C]=\triv$ and $C \nleq B$, there exists $D \in \mc{C}$ such that $\grp{A,B} \le D$.\label{item:joinable}
		\item For every $A\in \mc{C}$ there is $g\in H$ and a homomorphism $(-)^\bullet: \prod_{i \ge 1}A \rightarrow H$ such that
		\[
		\forall a \in A: (a,1,1,\dots)^\bullet = a
		\] 
		and
		\[
		\forall (a_i)_{i \ge 1} \in \prod_{i \ge 1}A: \; g(a_1,a_2,a_3,\dots)^\bullet g^{-1} = (1,a_1,a_2,\dots)^\bullet.
		\]
		 \label{item:translation}
		 
	\end{enumerate}
	We say $\mc{C}$ is \defbold{nomadic} for $H$ if conditions (A)--(B) are satisfied; that $H$ is \defbold{semi-transitive} on $\mc{C}$ if condition (C) is satisfied; that $\mc{C}$ is \defbold{joinable} over $H$ if condition (D) is satisfied; and that $\mc{C}$ is \defbold{shiftable} over $H$ if condition (E) is satisfied.
	We say $\mc{C}$ is a \defbold{compression family} for $H$ if it satisfies (A)--(C), a \defbold{joinable compression family} for $H$ if it satisfies (A)--(D) and a \defbold{shift-joinable compression family} for $H$ if it satisfies (A)--(E).
\end{defn}

The properties listed are derived from the typical case of interest, where $\mc{C} = \{\rist_G(Y) \mid Y \in \mc{C}^*\}$, where $\mc{C}^*$ is a $G$-invariant set of subsets of some set $X$ on which $G$ acts faithfully and compressibly.
Condition (D) is needed to reflect the fact that if the union of two subsets in $\mc{C}^*$ is not $X$ then it is an element of $\mc{C}^*$.
 See Lemma~\ref{lem:micro-supported_family} and Corollary~\ref{cor:micro-supported_compression} for sufficient conditions to obtain a (joinable) compression family this way.

 Conditions (A)--(D) imply a weak version of Condition (E): for every $A\in \mc{C}$ there is $g\in H$ such that $g^nAg^{-n}$ commutes with $A$ for all $n>0$ (see Lemma~\ref{lem:leafless}).  
 Condition (E) is inspired by the proof of Tits' simplicty theorem for groups acting on trees, where $H$ contains an infinite direct \emph{product} (as opposed to sum) of conjugates of $A\in\mc{C}$.

\

We now come to our first main theorem.

\begin{thm}[{See Section~\ref{sec:compressible_splittable}}]\label{thm:compressible_splittable_union}
Let $G$ be a group, $\mc{C}$ a joinable compression family for $G$, and $H = \grp{\mc{C}}$.
\begin{enumerate}[(i)]
	\item The group $\Der(H)$ is the monolith of $G$.
	\item If $H$ is semi-transitive on $\mc{C}$, then $\Der(H)$ is simple.
	\item\label{item:joinable_with_translation_implies_perfect} If $\mc{C}$ is a shift-joinable compression family for $H$, then $\Der(H) = H$ is the monolith of $G$ and is simple.
\end{enumerate}
\end{thm}

Theorem~\ref{thm:compressible_splittable_union} yields a statement about compressible micro-supported actions in the topological sense, as follows.

\begin{cor}[{Compare \cite[Proposition~4.6]{LeBoudecURS}}]\label{cor:fully_compressible_splittable}
Let $X$ be an infinite Hausdorff space, let $G \le \Homeo(X)$, let $\mc{C}$ be a $G$-invariant collection of non-empty open subsets of $X$ and let $H = \grp{\rist_G(Y) \mid Y \in \mc{C}}$.  Suppose that the elements of $\mc{C}$ are all $H$-compressible, and that given $Y,Z \in \mc{C}$, either $Y \cup Z$ is dense or $Y \cup Z \subseteq W$ for some $W \in \mc{C}$.  Then $G$ is almost simple with monolith $\Der(H)$, and the action of $\Der(H)$ is fully compressible.
\end{cor}

The following special cases recover some of the prior results from the literature.

\begin{cor}[{See Section~\ref{sec:tree}}]\label{cor:tree}
	Let $T$ be a tree and let $G \le \Aut(T)$ act geometrically densely on $T$; suppose $G^{++}$ is nontrivial.  Then $G$ is almost simple with monolith $\Der(G^{++})$; the action of $\Der(G^{++})$ is also geometrically dense.
\end{cor}

M\"{o}ller and Vonk in \cite{MollerVonk} study groups $G\leq \Aut(T)$ such that $G^{++}$ is non-trivial, and they call this latter condition having \emph{Property H} (for Half-trees). 
Theorem 6 of that paper states that a closed subgroup $G\leq \Aut(T)$ with Property H that acts geometrically densely has $\ol{G^{++}}$ as the closed monolith; in particular, $\ol{G^{++}}$ is topologically simple. 
Assuming that $G$ is locally compact (this is automatic if $T$ is locally finite) we can strengthen the conclusion of that theorem and of Corollary~\ref{cor:tree}.

\begin{cor}[{See Section~\ref{sec:open_monolith}}]\label{cor:tree:simpleG++}
Let $T$ be a tree and let $G \le \Aut(T)$ be closed, locally compact and act geometrically densely on $T$.  Then $G^{++}$ is trivial or abstractly simple.
\end{cor}

\begin{cor}[{See Section~\ref{sec:pw_full}}]\label{cor:Matui}
	Let $X$ be an  infinite zero-dimensional Hausdorff space and let $G$ be the piecewise full group of some group of homeomorphisms of $X$.  Suppose that the action of $G$ is compressible and let $H$ be the group generated by rigid stabilisers of $G$-compressible clopen sets.   Then $G$ is almost simple with monolith $\Der(H)$.  If $X$ is compact and $G$ is minimal, then $G = H$ and the action of $\Der(G)$ is fully compressible.
\end{cor}

	The abelianisation of $G$ as above is only known for certain cases and is usually and \emph{ad hoc} argument.
	For example, when $G$ is the full group of a shift of finite type, a formula for the abelianisation is given in \cite{MatuiSimple}. 
	For the non-discrete analogues studied in \cite{Lederle}, it is shown there that the abelianisation is the image of that of the full group of the corresponding shift of finite type. This image may be proper, since all Neretin groups are simple (\cite{Kapoudjian}) while not all the corresponding Higman--Thompson groups are simple. 
	Theorem~\ref{thm:compressible_splittable_union}\ref{item:joinable_with_translation_implies_perfect} cannot be used as is in these cases because  condition \ref{item:translation} may not be satisfied for the family of subgroups used to prove Corollary \ref{cor:Matui}.

In the absence of joinability, a compression family still puts restrictions on the normal subgroup structure of $G$: see Lemma~\ref{lem:double_commutator:subgroups:bis} and Proposition~\ref{prop:compressible_splittable}.  However, under these more general hypotheses it is not clear when $G$ has a simple normal subgroup. Indeed, the authors do not know the answer to the following:

\begin{que}
Let $X$ be an infinite compact zero-dimensional space, let $G \le \Homeo(X)$ act minimally, and suppose that $G$ is generated by elements of $G$-compressible support.  Is it always true that $\Der(G)$ is simple?
\end{que}

\subsection{Isometric actions and number of ends}

There is a notable difference in the large-scale geometry of the groups satisfying the hypotheses of Corollary~\ref{cor:Matui}, as opposed to those occurring in Corollary~\ref{cor:tree}.  While the latter clearly have a faithful action of general type on a hyperbolic space (namely, on the tree $T$), we find that in the case of piecewise full groups, there are significant obstructions to having interesting actions on hyperbolic spaces.  A group $G$ has \defbold{property $(\mathrm{NGT})$} if no action of $G$ by isometries on a hyperbolic space is of general type; see Definition~\ref{def:hyperbolic} for details.

By adapting arguments of Balasubramanya, Fournier-Facio and Genevois (\cite{BFG}, \cite{Genevois}), we obtain a sufficient condition in terms of strong compression families for $G$ to have property $(\mathrm{NGT})$.  Here is a special case for piecewise full groups.

\begin{thm}[{See Corollary~\ref{cor:alternatable_NGT}}]\label{thm:alternatable_NGT}
	Let $X$ be a perfect zero-dimensional compact Hausdorff space, let $G$ be a piecewise full group of homeomorphisms of $X$, and let $\Der(G) \le G_0 \le G$.  Then $G_0$ has property $(\mathrm{NGT})$. 
\end{thm}

For compactly generated locally compact groups, including finitely generated discrete groups, having more than one end is associated with certain loxodromic actions on trees.  In particular, we see the following.

\begin{cor}\label{cor:alternatable_NGT:ends}
Let $G$ and $G_0$ be as in Theorem~\ref{thm:alternatable_NGT}, such that $G_0$ is equipped with a locally compact group topology.  Suppose $G_0$ is compactly generated with no non-trivial continuous homomorphism to $\Rb$.  Then $G_0$ has at most one end.
\end{cor}

\subsection{A sufficient condition for a normal subgroup to be open}

The second main theorem concerns group topology, in particular the structure of \tdlc groups.  Given a topological group $G$ acting continuously on a compact zero-dimensional space $X$, we say the action is \defbold{locally decomposable} if for every partition $\mc{P}$ of $X$, the subgroup $\grp{\rist_G(P) \mid P \in \mc{P}}$ is open in $G$.
	
\begin{thm}[{See Section~\ref{sec:open_monolith}}]\label{thm:loc_decomp_compressible}
	Let $G$ be a \tdlc group acting faithfully, minimally and locally decomposably on a compact zero-dimensional space $X$, and let $A$ be a normal subgroup of $G$ (not assumed to be closed).  Suppose that the action of $A$ on $X$ is compressible.  Then $A$ is open in $G$ and the action of $A$ on $X$ is minimal and locally decomposable.
\end{thm}
	
	Theorem~\ref{thm:loc_decomp_compressible} applies to many almost simple \tdlc groups.
	For instance, in \cite{GR-generaltheory}, it is shown that a \tdlc group $G$ of homeomorphisms of a compact zero-dimensional space $X$, induces a unique \tdlc topology of its piecewise full group $\Full(G)$ if $G$ is also locally decomposable. 
	If $G$ is compactly generated and acts compressibly and minimally on $X$ then the monolith $\Al(G)$ is compactly generated (and simple), but it is not \emph{a priori} a closed subgroup of $\Full(G)$. 
	Theorem~\ref{thm:loc_decomp_compressible} guarantees that $\Al(G)$ is in fact open in $\Full(G)$ and in particular a \tdlc group that is compactly generated and abstractly simple. 
	If $G$ is further assumed non-discrete then $\Al(G)$ is a group in the class $\ms{S}$ of non-discrete \tdlc groups that are compactly generated and topologically simple. 
	This class is important in the general theory of \tdlc groups (see \cite{CRW-Part2}) and further consequences of the above result for this theory are elaborated in \cite{GR-generaltheory}, some of which require the next result.
	  We recall the class $\ms{R}$ introduced in \cite{CRW-DenseLC}.
		
	\begin{defn}
	We say a \tdlc group $G$ is \defbold{regionally expansive} if there is a compactly generated open subgroup $H$ of $G$ and an identity neighbourhood $U$ in $H$ such that $\bigcap_{h \in H}hUh\inv = \triv$.
	
	Write $\TMon(G)$ for the intersection of all non-trivial closed normal subgroups of $G$.  A \tdlc group $G$ is \defbold{robustly monolithic} if $\TMon(G)$ is non-discrete, regionally expansive, and topologically simple.
	
	Write $\ms{R}$ for the class of robustly monolithic groups.
	\end{defn}
	
	The class $\ms{R}$ in particular includes any \tdlc group $G$ with a closed normal subgroup $S$, such that $\CC_G(S) = \triv$ and $S$ is non-discrete, compactly generated and topologically simple.
	
	\begin{cor}[{See Section~\ref{sec:robustly_monolithic}}]\label{cor:R_open_monolith}
	Let $G \in \ms{R}$ and suppose $G$ has an open subgroup of the form $A \times B$ where $A$ and $B$ are non-trivial closed subgroups.  Then $\TMon(G)$ is open in $G$ and abstractly simple.
	\end{cor}

\section{Compression families and normal subgroups}

In this section we obtain results on the normal subgroup structure of groups with compressible action, including Theorem~\ref{thm:compressible_splittable_union}.

\subsection{Preliminaries}

First let us note some equivalent formulations of almost simplicity that are easily verified and that we will use without further comment.

\begin{lem}
Let $G$ be a group.  The following are equivalent:
\begin{enumerate}[(i)]
\item $G$ has a non-abelian simple monolith $M$;
\item $G$ has a simple normal subgroup $M$ with $\CC_G(M) = \triv$;
\item $G$ has a non-abelian normal subgroup $M$, such that every non-trivial subgroup of $G$ that is normalised by $M$, contains $M$.
\end{enumerate}
\end{lem}

Here are some well-known facts about commutators and normalisers that we will use later.

\begin{lem}\label{lem:commutator_product}
Let $G$ be a group, let $H,K \le G$ with $H$ normalizing $K$ and let $h_1,h_2 \in H$.
\begin{enumerate}[(i)]
\item Suppose there is $k \in K$ such that $kh\inv_1 k\inv$ commutes with $h_2$.  Then $[h_1,h_2] \in K$.
\item Let $g \in G$.  If $[h_1,g],[h_2,g] \in K$, then $[h_1h_2,g] \in K$.
\end{enumerate}
\end{lem}
	
\begin{proof}
(i)
We see that $[h_1,h_2] = [h_1(kh\inv_1 k\inv),h_2] = [[h_1,k],h_2]$.  Since $K$ is normalised by $h_1$ and $h_2$, we have $[h_1,k] \in K$ and hence $[[h_1,k],h_2] \in K$, so $[h_1,h_2] \in K$ as required.

(ii)
Let $h = h_1h_2$.  We have
\[
[h,g] = [h_1h_2,g] = h_1[h_2,g]h\inv_1 [h_1,g];
\]
if $[h_1,g],[h_2,g] \in K$, we conclude that $[h,g] \in K$.
\end{proof}

We collect together some properties of compression families of subgroups that will be useful later.

\begin{lem}\label{lem:nondegen}
Let $G$ be a non-trivial group and let $\mc{C}$ be a family of subgroups of $G$ that is invariant under conjugation and satisfies conditions \ref{item:nonabelian} and \ref{item:disjointtranslates} for $H = G$.  Then there exist distinct $A,B,C \in \mc{C}$ such that $[A,C]=\triv$ and $C \nleq B$.
\end{lem}

\begin{proof}
We suppose for a contradiction that for all distinct $A,B,C \in \mc{C}$, if $[A,C]=\triv$ then $C \le B$.  

Let $g \in G \setminus \{1\}$ and let $A \in \mc{C}$ such that $[A,gAg\inv]=\triv$.  Since $A$ is not abelian we have $A \neq gAg\inv$, so by hypothesis, for all $B \in \mc{C} \setminus \{A,gAg\inv\}$ we have $gAg\inv \le B$.  Now let $a \in A \setminus \{1\}$ and $B \in \mc{C}$ such that $[B,aBa\inv]=\triv$.  Since $a$ normalises both $A$ and $gAg\inv$ we have $B \in \mc{C} \setminus \{A,gAg\inv\}$.  Then $gAg\inv \le B$, and hence also $gAg\inv \le aBa\inv$.  But then $B$ cannot commute with $aBa\inv$ since $gAg\inv$ is non-abelian, giving the desired contradiction.
\end{proof}

\begin{lem}\label{lem:leafless}
Let $G$ be a non-trivial group, let $\mc{C}$ be a joinable compression family of subgroups of $G$ and let $H$ be a normal subgroup of $G$ that is semi-transitive on $\mc{C}$.
\begin{enumerate}[(i)]
\item For all $A \in \mc{C}$ there exist $B_1,B_2 \in \mc{C}$ such that $B_1,B_2 \le A$ and $[B_1,B_2]=\triv$.
\item For all $A \in \mc{C}$ there exists $g \in H$ such that $[g^mAg^{-m},g^nAg^{-n}] = \triv$ for all distinct $m,n \in \Zb$.

\item Let $A_1,A_2 \in \mc{C}$.  Then there exist $B_1,B_2 \in \mc{C}$ such that
\[
[A_1,B_1] = [B_1,B_2] = [B_2,A_2] = \triv.
\]
\item Given $A \in \mc{C}$ there is $B \in \mc{C}$ such that $A \le B$ and $\Der(A) \le \Der(B \cap H)$.  Consequently, $A \cap H$ is never abelian for $A \in \mc{C}$.
\end{enumerate}
\end{lem}

\begin{proof}
We use Lemma~\ref{lem:nondegen} to obtain distinct $A,B,C \in \mc{C}$ such that $[A,C]=\triv$ and $C \nleq B$.  Since $\mc{C}$ is joinable, there exists $D \in \mc{C}$ such that $\grp{A,B} \le D$.  In particular, we have $A < D$ or $B < D$ (or both).  By semi-transitivity of the action on $\mc{C}$, it follows that every $A \in \mc{C}$ properly contains an element of $\mc{C}$.

Next, let $g \in G \setminus \{1\}$ and $A \in \mc{C}$ such that $[A,gAg\inv]=\triv$, and let $B \in \mc{C}$ such that $B < A$.  Then $gBg\inv$ does not contain $gAg\inv$, which commutes with $A$; hence by joinability again, there is $C \in \mc{C}$ such that $\grp{gBg\inv,A} \le C$.  In particular, $C$ contains both $B$ and $gBg\inv$, and we have $[B,gBg\inv]=\triv$.  By semi-transitivity, every $A \in \mc{C}$ contains a conjugate of $C$, and part (i) follows.

For (ii), let $A\in \mc{C}$ and pick $A_0\in \mc{C}$. 
By part (i), there are $B_1, B_2\in \mc{C}$ that commute and are contained in $A_0$.
By \ref{item:semitransitive}, there are $f, h\in H$ such that $fAf^{-1}\leq B_1$ and $hA_0h^{-1}\leq B_2$. 
Taking $g=f^{-1}hf$, we have that $[g^nAg^{-n},g^mAg^{-m}]=1$ for every distinct $m,n\in \Zb$.

For (iii), suppose first that for all $C \in \mc{C}$ that commute with $A_1$, we have $C \le A_2$.
Then by (ii) there are conjugates $B_i\in \mc{C}$ of $A_i$ that commute with $A_i$ for $i=1,2$.
Since $B_1$ commutes with $A_1$, it is contained in $A_2$, so $[B_1,B_2]\leq [A_2,B_2]=\triv$.
Now suppose instead that there is $C\in \mc{C}$ that commutes with $A_1$ but  $C\nleq A_2$. 
Then by joinability there is $D\in \mc{C}$ such that $\grp{A_1, A_2}\le D$.
Now, by (ii) there are conjugates $B_1,B_2\in \mc{C}$ of $D$ that commute with each other and with $D$, and these verify the statement.

For (iv), take $A\in \mc{C}$ and $g \in H$ as in part (ii) and write $A_i = g^iAg^{-i}$.  Applying \ref{item:joinable} twice, there is $B \in \mc{C}$ such that $\grp{A_0,A_1,A_2} \le B$.  Given $a,b \in A$, we have
\[
[a,b] = [[a,g],[b,g^2]],
\]
showing that $[a,b]$ belongs to the derived group of $B \cap H$.
Since $A$ is not abelian we deduce that $B \cap H$ is not abelian; by semi-transitivity it follows that $A \cap H$ is non-abelian for all $A \in \mc{C}$.
\end{proof}

\subsection{Micro-supported actions}

Micro-supported actions are natural sources of examples of families of subgroups as in Definition~\ref{def:msc:subgroups}. 

\begin{defn}
Let $G$ be a (topological) group acting faithfully on a Hausdorff space $X$.  We say the action is \defbold{micro-supported} if for all non-empty open $Y \subseteq X$, the rigid stabiliser $\rist_G(Y)$ is non-trivial.
\end{defn}

The next lemma gives sufficient conditions to achieve the conditions in Definition~\ref{def:msc:subgroups} from a permutational perspective.

\begin{lem}\label{lem:micro-supported_family}
Let $G$ be a group acting faithfully on a set $X$ and $H\le G$.
 Let $\mc{C}^*$ be a $G$-invariant set of subsets of $X$
 and $\mc{C} = \{\rist_G(Y) \mid Y \in \mc{C}^*\}$.  
 Consider the following properties:
\begin{enumerate}[(i)]
		\item For all $Y \in \mc{C}^*$ the rigid stabiliser $\rist_G(Y)$ is non-trivial.
		\item For every $g \in G \setminus \{1\}$ there exists $Z \in \mc{C}^*$ such that $Z \cap gZ = \emptyset$.
		\item For all $Y,Z \in \mc{C}^*$ there exists $g \in H$ such that $gY \subseteq Z$.
		\item For all $Y,Z \in \mc{C}^*$, if $\rist_G(Y) \nleq \rist_G(Z)$ then there exists $W \in \mc{C}^*$ such that $W \subseteq Y$ and $\rist_G(Z \cap W)=\triv$.
		\item For all $Y,Z \in \mc{C}^*$, if $\rist_G(Y \cap W_1) = \rist_G(Z \cap W_1) = \triv$ for some $W_1 \in \mc{C}^*$, then $Y \cup Z \subseteq W_2$ for some $W_2 \in \mc{C}^*$.
\end{enumerate}
If (i)--(ii) hold then $\mc{C}$ is a nomadic family of subgroups of $G$. 
If (i)--(iii) hold then $\mc{C}$ is a compression family for $H$.
If (i)--(v) hold then $\mc{C}$ is a joinable compression family for $H$.
\end{lem}

\begin{proof}
We may suppose that hypotheses (i)--(ii) hold.

Given $g \in G \setminus \{1\}$ and $Z \in \mc{C}^*$ such that $Z \cap gZ = \emptyset$, then
\[
B \cap gBg\inv=[B,gBg\inv]=\triv
\]
where $B = \rist_G(Z)$. 

Moreover, given $1\neq a\in A=\rist_G(Y)$, there is $B=\rist_G(Z)$ such that $Z\cap aZ=\emptyset$. 
Thus $Z\subset Y$ and $B< A$, which means that $a$ is not central in $A$, so $A$ is not abelian. 
We deduce that $\mc{C}$ is nomadic for $G$.
It is clear that if (iii) holds, then $H$ is semi-transitive on $\mc{C}$.

Finally, suppose that (i)--(v) hold and that  $Y_1,Y_2,Y_3 \in \mc{C}^*$ are such that $\rist_G(Y_1)$ commutes with $\rist_G(Y_3)$ and $\rist_G(Y_3) \nleq \rist_G(Y_2)$.  By (iv) there exists $W \in \mc{C}^*$ such that $W \subseteq Y_3$ and $\rist_G(Y_2 \cap W)=\triv$. 
As $\rist_G(Y_1)$ commutes with $\rist_G(Y_3)$, the rigid stabiliser $\rist_G(Y_1 \cap W)$ is also abelian, hence trivial. 
Thus (v) ensures that there exists $W_2 \in \mc{C}^*$ such that $Y_1 \cup Y_2 \subseteq W_2$, from which it follows that $\grp{\rist_G(Y_1),\rist_G(Y_2)} \le \rist_G(W_2)$. 
 Thus $\mc{C}$ is joinable.
\end{proof}

The following special cases appear in examples of interest. 
		
\begin{cor}\label{cor:micro-supported_compression}
Let $X$ be a Hausdorff space and let $G \le \Homeo(X)$.  Let $\mc{C}^*$ be a $G$-invariant set of non-empty subsets of $X$, such that for all $Y \in \mc{C}^*$ and every non-empty open $Z \subseteq X$, there is $g \in G$ such that $gY \subseteq Z$.  Suppose that $\rist_G(Y) \neq \triv$ for every $Y \in \mc{C}^*$.
\begin{enumerate}[(i)]
\item The set $\mc{C}^*$ satisfies hypotheses (i)--(iv) of Lemma~\ref{lem:micro-supported_family}.  Consequently, $\mc{C} = \{\rist_G(Y) \mid Y \in \mc{C}^*\}$ is a compression family of subgroups of $G$.
\item Suppose in addition that for all $Y,Z \in \mc{C}^*$, either the union of the interiors of $Y$ and $Z$ is dense, or $Y \cup Z$ is contained in an element of $\mc{C}^*$.  Then $\mc{C}$ is a joinable compression family of subgroups of $G$.
\end{enumerate}
\end{cor}
\begin{proof}
	
	(i). Note that $X$ being Hausdorff implies that $\supp(g)=\{x\in X\mid g(x)\neq x\}$ is open for every $g\in G$ and so $\rist_G(Y)=\rist_G(Y^o)$ for every $Y\in \mc{C^*}$. 
	In particular, $\rist_G(Y)\neq \triv$ for all $Y\in \mc{C}^*$ implies that every $Y$ has non-empty interior. 
	This together with the other assumption in the statement easily implies hypotheses (i)--(iii) of Lemma \ref{lem:micro-supported_family}.
	Moreover, since every $Y\in \mc{C}^*$ is mapped into every open subset of $X$, no element of $\mc{C}^*$ can be dense in $X$.  
	To see that (iv) holds, suppose that $\rist_G(Y)\nleq\rist_G(Z)$ for some $Y, Z\in\mc{C}^*$. 
	Then there is some $r \in \rist_G(Y)$ and $x \in Y \setminus Z^o$ such that $rx \neq x$.  We see that indeed there is a neighbourhood $Y_0$ of $x$ such that $Y_0 \cap Z^o = \emptyset$ and $rY_0 \cap Y_0 = \emptyset$, ensuring $Y_0 \subseteq Y$.  Then there is $g \in G$ such that $g(Y) \subseteq Y_0$.
	Taking $W:=g(Y)$ verifies hypothesis (iv).
	
	(ii). To see that hypothesis (v) is verified, let $Y, Z\in\mc{C}^*$ such that $\rist_G(Y\cap W_1)=\rist_G(Z\cap W_1)$ for some $W_1\in \mc{C}^*$. 
	If $Y\cap W_1$ has non-empty interior, there is some $g\in G$ taking $Y$ into $Y\cap W_1$, so $1\neq g\rist_G(Y)g^{-1}\leq \rist_G(Y\cap W_1)$, a contradiction. 
	Thus $Y\cap W_1$ and similarly $Z\cap W_1$ have empty interior, so $Y^o\cup Z^o$ cannot be dense in $X$. 
\end{proof}

\subsection{Compression families and the monolith}\label{sec:compressible_splittable}

We now establish some consequences of the existence of a compression family of subgroups for normal subgroup structure.
	
\begin{lem}\label{lem:double_commutator:subgroups:bis}
	Let $G$ be a group and let $\mc{C}$ be a compression family of subgroups of $G$.
	\begin{enumerate}[(i)]
		\item Given $A,B \in \mc{C}$ there is $g \in G$ such that $[A,gBg\inv]=\triv$.
		\item Let $M$ be the monolith of $G$.  Then for all $B \in \mc{C}$, we have $\Der(B) \le M$ and there is $g \in M$ such that $[B,gBg\inv]=\triv$.  In particular, $M$ is the non-abelian group
		\[
		M = \grp{\Der(A) \mid A \in \mc{C}}.
		\]
	\end{enumerate}
\end{lem}
	
\begin{proof}
		For (i), given $g_0 \in G \setminus \{1\}$, there is some $C \in \mc{C}$ such that $[C,g_0Cg\inv_0]=\triv$ by condition (A); clearly also $[C',g_0C'g\inv_0]=\triv$, where $C'$ is any subgroup of $C$.  By condition (B), after conjugating $g_0$ we can in fact choose $C'$ to be any element of $\mc{C}$; in particular, there is $g_1 \in G \setminus \{1\}$ such that $[A,g_1Ag\inv_1]=\triv$.  There is also $g_2 \in G$ such that $g_2Bg\inv_2 \le A$, and hence $[A,gBg\inv]=\triv$ where $g = g_1g_2$.
		
		Let $N$ be a non-trivial normal subgroup of $G$.  By Lemma~\ref{lem:commutator_product}(i) and condition (A) there exists $B \in \mc{C}$ and $g \in N$ such that $\Der(B) \le N$ and $[B,gBg\inv]=\triv$.  For a fixed normal subgroup $N$, these properties of $B$ are clearly $G$-invariant and preserved by passing to a subgroup.  Thus by condition (B), we see that for all $B \in \mc{C}$ we have $\Der(B) \le N$ and $g \in N$ such that $[B,gBg\inv]=\triv$.  Since there is no dependence on $N$, it follows that $\Der(B) \le M$ for all $B \in \mc{C}$; the other inclusion follows from the fact that  $\{D(A) \mid A\in\mc{C}\}$ is a $G$-invariant set of subgroups of $G$.
		Since $\Der(B)$ is not normal in $M$, we deduce that $M$ is non-abelian.    Since $M$ is itself a non-trivial normal subgroup, for all $B \in \mc{C}$ there exists $g \in M$ such that $[B,gBg\inv]=\triv$.  This completes the proof of (ii).
\end{proof}

If in addition the group $\grp{\mc{C}}$ generated by $\mc{C}$ also acts semi-transitively on $\mc{C}$, we can characterise $\Der(\grp{\mc{C}})$ among the normal subgroups of $G$.

\begin{prop}\label{prop:compressible_splittable}
Let $G$ be a group with a compression family $\mc{C}$ and let $H = \grp{\mc{C}}$.  Suppose that $H$ acts semi-transitively on $\mc{C}$ by conjugation.  Then $\Der(H)$ acts transitively on every $H$-orbit on $\mc{C}$.  Moreover, $K = \Der(H)$ is the smallest non-trivial $H$-invariant subgroup $K$ of $G$ with the following property:
	
	$(*)$ There is $A \in \mc{C}$, such that for all $g \in G$ there exists $k \in K$ such that $kAk\inv$ commutes with $gAg\inv$.
\end{prop}

\begin{proof}
Let $S = \bigcup \mc{C}$.  Given $A \in \mc{C}$, we must show that every $H$-conjugate of $A$ is a $\Der(H)$-conjugate of $A$.  First consider $g \in B \in \mc{C}$.  By Lemma~\ref{lem:double_commutator:subgroups:bis}(i) there is $h \in H$ such that $[A,hBh\inv]=\triv$; in particular, $hg\inv h\inv$ commutes with $A$, so that $[g,h]A[h,g] = gAg\inv \in \mc{C}$.  Given a general element $g \in H$, for some $n \in \Nb$ we can write $g = g_1g_2 \dots g_n$, where $g_i \in S$ for $1 \le i \le n$.  
Repeating the above inductively with $g_i\dots g_nAg_n^{-1}\dots g_i^{-1}$ in place of $A$, we obtain $h_1,h_2,\dots,h_n \in H$ such that
		\[
		g_1g_2 \dots g_nAg\inv_n \dots g\inv_2 g\inv_1 = [g_1,h_1][g_2,h_2] \dots [g_n,h_n]A[h_n,g_n] \dots [h_2,g_2][h_1,g_1],
		\]
showing that $gAg\inv$ is a $\Der(H)$-conjugate of $A$ as claimed.  Applying Lemma~\ref{lem:double_commutator:subgroups:bis}(i) again, we see that $\Der(H)$ satisfies $(*)$.

Now let $K$ be a non-trivial $H$-invariant subgroup of $G$ satisfying $(*)$.  Since $K$ is $H$-invariant, the condition put on $A$ in the statement of $(*)$ is stable under replacing $A$ with an $H$-conjugate of a subgroup of $A$; thus by condition (B), it is the case that for all $A \in \mc{C}$ and $g \in H$, there exists $k \in K$ such that $[kAk\inv,gAg\inv] = \triv$.  In turn, given $B \in \mc{C}$, then every $H$-conjugate of $B$ is contained in a $H$-conjugate of $A$; thus for all $A,B \in \mc{C}$ and $g \in H$, there exists $k \in K$ such that $[kAk\inv,gBg\inv] = \triv$.  In particular, taking $g=1$, then given $a \in A$ and $b \in B$ we have $k \in K$ such that $ka\inv k\inv$ commutes with $b$, ensuring
\[
[a,b] = [[a,k],b] \in K.
\]
Thus $\ngrp{[s,t] \mid s,t \in S} \le K$.  By Lemma~\ref{lem:commutator_product}(ii), it follows that
		\[
		\ngrp{[s,t] \mid s,t \in S} = \ngrp{[s,t] \mid s,t \in \grp{S}} = \Der(H),
		\]
so $\Der(H) \le K$, proving that $\Der(H)$ is the smallest normal subgroup of $G$ satisfying $(*)$.
\end{proof}

We now prove the first main theorem from the introduction.

\begin{proof}[Proof of Theorem~\ref{thm:compressible_splittable_union}]
It is clear that $\Der(H) \neq \triv$.
Let $M$ be the intersection of all non-trivial normal subgroups of $G$; note that $M \le \Der(H)$.  
By Lemma~\ref{lem:double_commutator:subgroups:bis}(ii) we have $\Der(A) \le M$ for all $A \in \mc{C}$, so in particular $M \neq \triv$.
		
		Similar to the proof of Proposition~\ref{prop:compressible_splittable}, to show $\Der(H) = M$ it is enough to show, given $A,B \in \mc{C}$ and $a \in A, b \in B$, that $[a,b] \in M$.  If $A \le B$ this has been shown in Lemma~\ref{lem:double_commutator:subgroups:bis}(ii), so assume $A \nleq B$.  By Lemma~\ref{lem:double_commutator:subgroups:bis}(ii) there exists $h \in M$ such that $[A,hAh\inv] = \triv$.  Then $A$ commutes with $hAh\inv$ but is not contained in $B$.  By joinability, there exists $C \in \mc{C}$ such that $\grp{hAh\inv, B} \le C$.  In particular, we see that $[hAh\inv,B] \subseteq \Der(C) \le M$, so $[hah\inv,b] \in M$ for all $a \in A$ and $b \in B$. 
		Since $[h,b], [h\inv,b]\in M$, we deduce applying Lemma~\ref{lem:commutator_product}(ii)  twice
		that $[h\inv hah\inv h, b]=[a,b]\in M$ for all $a \in A$ and $b \in B$ as claimed, and hence that $\Der(H) = M$, completing the proof of (i).
		
		For part (ii), we assume that $H$ is semi-transitive on $\mc{C}$. 
		Let $M^*$ be the intersection of all non-trivial normal subgroups  of $\Der(H)$. 
		By Proposition~\ref{prop:compressible_splittable} we know that the action of $\Der(H)$ on $\mc{C}$ is semi-transitive, and then by Lemma~\ref{lem:leafless}, the intersection $A \cap \Der(H)$ is non-abelian for all $A \in \mc{C}$.    A similar argument to Lemma~\ref{lem:double_commutator:subgroups:bis}(ii) now shows that, for all $A \in \mc{C}$, we have $\Der(A \cap \Der(H)) \le M^*$.  Thus $M^*$ is non-trivial.  By construction, $M^*$ is characteristic in $\Der(H)$, hence normal in $G$; since $\Der(H)$ is already the monolith of $G$, we deduce that $M^* = \Der(H)$, so $\Der(H)$ is simple.
		
		For (iii), it suffices to show that $\Der(H)\geq H=\grp{\mc{C}}$.
		Given any $A\in \mc{C}$, by condition \ref{item:translation} there is $g\in H$ and a homomorphism $(-)^\bullet: \prod_{i \ge 0}A \rightarrow H$, such that $A$ is the image of the first factor, and $g$ acts on the image of $(-)^\bullet$ as a shift map.  We then obtain $a$ as the commutator $bgb^{-1}g^{-1}$ where $b = (a,a,a,\dots)^\bullet$.
\end{proof}

Combining Corollary~\ref{cor:micro-supported_compression} with Theorem~\ref{thm:compressible_splittable_union} yields Corollary~\ref{cor:fully_compressible_splittable}.

\subsection{Compressible piecewise full groups}\label{sec:pw_full}

We now consider piecewise full groups.  Note that for any group $G$ of homeomorphisms we have $\Full(\Full(G)) = \Full(G)$, so in studying piecewise full groups, there is no loss of generality in assuming $G = \Full(G)$.

Given $g \in G$ and a clopen subset $Y$ of $X$ such that $Y \cap gY = \emptyset$, write $\sigma(g,Y)$ for the following homeomorphism:
	\[
	\sigma(g,Y)x = 
	\begin{cases}
		gx &\mbox{if} \;  x \in Y \\ 
		g\inv x &\mbox{if} \;  x \in gY\\
		x &\text{otherwise}
	\end{cases}.
	\]

\begin{lem}\label{lem:pw_full:split}
Let $X$ be a zero-dimensional Hausdorff space with at least three points and let $G \le \Homeo(X)$ such that $G = \Full(G)$.  Then $G = SS$ where $S$ is the union of all rigid stabilisers in $G$ of proper clopen subsets of $X$.
\end{lem}

\begin{proof}
Let $g \in G \setminus \triv$.  Then there exists $x \in X$ such that $gx \neq x$, and hence a clopen neighbourhood $Y$ of $x$ such that $Y \cap gY = \emptyset$.  By choosing $Y$ small enough we may ensure that $Y \cup gY \neq X$.  Let $h = \sigma(g,Y)$; since $G = \Full(G)$ we have $h \in G$. 
 We see that in fact $h \in S$ (since $Y \cup gY$ is a proper clopen subset of $X$) and also $gh\inv \in S$ (since $gh\inv$ fixes $gY$ pointwise), so $g \in SS$. 
\end{proof}

\begin{lem}\label{lem:pw_full}
Let $X$ be an infinite zero-dimensional Hausdorff space and let $G \le \Homeo(X)$ such that $G = \Full(G)$.  Suppose the action of $G$ is compressible.
\begin{enumerate}[(i)]
\item Given $G$-compressible clopen sets $Y$ and $Z$, the union $Y \cup Z$ is either the whole space or $G$-compressible.
\item Let $Y\subset X$ be a $G$-compressible clopen set, $\mc{C}=\{\rist_G(gY)\mid g\in G\}$ and $H=\grp{\mc{C}}$.
  Then $H \unlhd G$ and every $G$-translate of $Y$ is $H$-compressible.
\item \label{item:fullycomp_implies_generatedbyrists}If $X$ is compact and the action of $G$ is minimal, then the action is fully compressible and $G =\grp{\rist_G(Y)\mid Y \text{ is } G\text{-compressible}}$.
\end{enumerate}
\end{lem}

\begin{proof}
For (i), let $Y$ and $Z$ be $G$-compressible clopen sets.  By replacing $Z$ with $Z \setminus Y$, we may assume $Y$ and $Z$ are disjoint.  Let $W = (Y \cup Z)^c$; we may assume $W$ is non-empty.  Since $Y$ and $Z$ are $G$-compressible there exist $g_1,g_2,g_3 \in G$ such that $g_1Y \subseteq W$ and such that $g_2Y$ and $g_3Z$ are disjoint subsets of $g_1Y$.  Now let $g = g\inv_1h$ where 
\[
h = \sigma(g_2,Y)\sigma(g_3,Z).
\]
By construction we see that $h \in G$, so $g \in G$.  We also see that $h(Y \cup Z) \subseteq g_1Y$, so $g(Y \cup Z) \subseteq Y$.  Since $Y$ is $G$-compressible, we deduce that $Y \cup Z$ is also $G$-compressible, proving (i).

For (ii), it is clear from the definition that $H$ is normal in $G$.
By part (i) we see that if $U$ and $V$ are $G$-compressible clopen sets and $U \cup V \neq X$, then $U \cup V$ is contained in a $G$-translate of $Y$ and hence $\rist(U \cup V) \le H$. 
Let $Z$ be a $G$-translate of $Y$ and $U$ a non-empty open set. 
Pick $g \in G$ such that $gZ \subseteq U$; by choosing a smaller $U$ if necessary we may ensure $Z \cup gZ \neq X$.  Then there is $f \in G$ such that $fZ$ is disjoint from $Z \cup gZ$ and $Z \cup gZ \cup fZ \neq X$.  
Taking  $h =\sigma(gf\inv,fZ)\sigma(f,Z) \in H$, we see that $gZ = hZ$, so $Z$ is $H$-compressible.

For (iii), let $Y$ be a $G$-compressible clopen set.
Since $X$ is compact and the action is minimal, there is a cover of $X$ by finitely many $G$-translates $Y,g_1Y,\dots,g_nY$ of $Y$; suppose $n$ is minimised.  Then $Z = \bigcup^n_{i=1}g_iY$ is not dense, so by part (i), $Z$ is $G$-compressible.
Now given non-empty, non-dense clopen subsets $W_1, W_2 \subseteq X$, we see that $W_2$ contains a $G$-translate of $W_1$ as follows:  
Since $Y$ is $G$-compressible there is $g \in G$ such that $gY \subseteq X\setminus W_1$; since $X = Y \sqcup Z$, we have $W_1 \subseteq gZ$; and since $Z$ is $G$-compressible there is $k \in G$ such that
\[
kg\inv W_1 \subseteq kZ \subseteq W_2,
\]
showing that the action is fully compressible. 
Lemma~\ref{lem:pw_full:split} now implies the last claim.
\end{proof}

We deduce the following from Lemma~\ref{lem:pw_full}, Corollary~\ref{cor:micro-supported_compression} and Theorem~\ref{thm:compressible_splittable_union}.

\begin{cor}
Let $X$ be an infinite zero-dimensional Hausdorff space, let $G$ be the piecewise full group of some group of homeomorphisms of $X$, and suppose some non-empty clopen $Y \subseteq X$ is $G$-compressible; let $\mc{C} := \{\rist_G(gY) \mid g \in G\}$.
  Then $\mc{C}$ is a joinable compression family for both $G$ and $H = \grp{\mc{C}}$;
   consequently, $G$ is almost simple with monolith $\Der(H)$.
\end{cor}

Corollary~\ref{cor:Matui} follows as a special case and essentially recovers \cite[Theorem~4.16]{MatuiSimple}, as follows.  The hypothesis of \cite[Theorem~4.16]{MatuiSimple} is that $G$ is an essentially principal \'{e}tale groupoid whose unit space is the Cantor space, such that $G$ is purely infinite and minimal.  However, Matui shows (\cite[Proposition~4.11]{MatuiSimple}) that for a minimal essentially principal \'{e}tale groupoid $G$ whose unit space is the Cantor space, $G$ is purely infinite if and only if $\Full(G)$ is fully compressible.  Thus the hypothesis of \cite[Theorem~4.16]{MatuiSimple} implies that $\Full(G)$ is fully compressible, and then Corollary~\ref{cor:Matui} implies that $\Der(\Full(G))$ is simple and the monolith of $\Full(G)$.

\subsection{Compressible actions coming from automorphisms of trees}\label{sec:tree}

For another source of examples for Theorem~\ref{thm:compressible_splittable_union}, we consider actions on trees.

\begin{defn}
We say a tree $T$ is of \defbold{general type} if some vertex of $T$ has at least three neighbours; in other words, we exclude the case that $T$ embeds isometrically in $\Rb$.  A group $G$ acting by automorphisms on a tree $T$ is \defbold{geometrically dense} if $G$ does not fix any end or preserve any proper subtree of $T$.
\end{defn}

\begin{lem}\label{lem:tree:geom_dense}
Let $T$ be a tree of general type and let $G \le \Aut(T)$.
\begin{enumerate}[(i)]
\item {\cite[Lemma~2.3]{ReidSmith}} $G$ is geometrically dense if and only if every half-tree of $T$ contains a $G$-translate of every other half-tree of $T$.
\item {\cite[Lemme~4.4]{Tits70}} If $G$ is geometrically dense, then so is every non-trivial normal subgroup of $G$.
\item {\cite[Lemma~3(iii)]{MollerVonk}} If $G$ does not preserve any proper subtree of $T$ then for every edge of $T$ there is a translation in $G$ whose axis contains that edge.
\end{enumerate}
\end{lem}

We now prove Corollary~\ref{cor:tree} using Theorem~\ref{thm:compressible_splittable_union} and Lemma~\ref{lem:tree:geom_dense}.

\begin{proof}[Proof of Corollary~\ref{cor:tree}]
By hypothesis $G$ is geometrically dense with $G^{++}$ non-trivial. 
This cannot occur unless $T$ is of general type, and in fact $T$ must be leafless and such that every ray of $T$ contains vertices with more than two neighbours. 
In particular, $\partial T$ is a Cantor space.
By construction, $G^{++}$ is normal in $G$ and by Lemma~\ref{lem:tree:geom_dense} (ii), the action of $G^{++}$ is geometrically dense.
Let $\mc{C}^*$ be the set of boundaries of half-trees of $T$.
Each half-tree $Y$ then has a rigid stabiliser $\rist_G(Y)$, which coincides with the rigid stabiliser of its boundary.
We can write $G^{++} = \grp{\mc{C}}$ where $\mc{C}$ is the set of rigid stabilisers of boundaries of half-trees; in particular, $\rist_G(Y)$ is non-trivial for some $Y \in \mc{C}^*$, and hence by Lemma~\ref{lem:tree:geom_dense}(i), in fact all elements of $\mc{C}$ are non-trivial.
Since $\mc{C}^*$ forms a base of clopen sets for the topology of $\partial T$, the hypotheses of Corollary \ref{cor:micro-supported_compression} hold and so $\mc{C}$ is a joinable compression family for $G^{++}$.

It follows from Theorem~\ref{thm:compressible_splittable_union}  that $G$ is almost simple with monolith $\Der(G^{++})$. 
The action of $\Der(G^{++})$ on $T$ is itself geometrically dense by Lemma~\ref{lem:tree:geom_dense}(ii).
\end{proof}

Corollary~\ref{cor:tree} immediately implies and strengthens \cite[Theorem~4.5]{LB-AP}.
It also recovers \cite[Theorem~6]{MollerVonk}, namely, the closure $\ol{G^{++}}$ of $G^{++}$ in the permutation topology is topologically simple.

The proof of the stronger Corollary \ref{cor:tree:simpleG++}  (that $G^{++}$ is perfect) is postponed to Section~\ref{sec:open_monolith}. 
This latter result recovers, in the case where the group is locally compact, Tits' simplicity theorem \cite{Tits70} and \cite[Theorem 7.3]{BanksElderWillis} where it is moreover assumed that the group of automorphisms of the tree has an \emph{independence property}, in essence, forcing \ref{item:translation} to hold for $\mc{C}$ the family of rigid stabilisers of half-trees.

Since the late 1990s, several generalisations of Tits' simplicity theorem have appeared in the literature to higher-dimensional spaces. 
They all follow Tits' blueprint and the bulk of the work in each is to show that the groups and chosen family of subgroups $\mc{C}$ satisfy conditions (A)--(E), so that an argument similar to the proof of Theorem \ref{thm:compressible_splittable_union} \ref{item:joinable_with_translation_implies_perfect} can be applied:

The first such examples appear in a paper by Haglund and Paulin \cite{HaglundPaulin}. 
It is shown that for an even polyhedral complex $P$ (a generalisation of cube complex) that is CAT(0) and Gromov-hyperbolic with at least 3 points on its boundary, if its group of automorphisms $G$ has the property that $G^{++}$ acts faithfully on $\partial P$, then $G^{++}$ is the simple monolith of $G$.
Here, $\mc{C}$ is the family of rigid stabilisers of half-spaces of $P$.

This is also the approach taken in \cite{Lazarovich}, where more general groups $G\leq \Aut(X)$, but where $X$ is a proper finite-dimensional irreducible (not a product of smaller ones) CAT(0) cube complex.
The same conclusion is arrived at, assuming that $G^{++}$ acts faithfully on $\partial X$, that there are at least 3 points on $\partial X$ that are fixed points of hyperbolic elements of $G$, that every half-space of $X$ contains $G$-orbit points arbitrarily far away from the defining hyperplane, and that $G$ has the analogue of Tits' property (P) that guarantees that \ref{item:translation} holds when taking $\mc{C}$ as the family of rigid stabilisers of half-spaces of $X$.

The results in \cite{Caprace-buildings} and \cite{deMedtsSilvaStruyve} are even stronger, for more restrictive families of groups. 
In these papers, right-angled buildings are considered; the group of all type-preserving automorphisms in the former, certain groups of type-preserving automorphisms prescribed by local actions on the panels of each type (mimicking the Burger--Mozes universal group construction for trees) in the latter. 
In both cases, if $G$ denotes the automorphism group in question, it is shown that $G^{++}$ is the simple monolith of $G$ by applying Theorem \ref{thm:compressible_splittable_union}\ref{item:joinable_with_translation_implies_perfect} where $\mc{C}$ consists of rigid stabilisers of \emph{wings} (analogues of half-trees, see \cite[Section 2.3]{deMedtsSilvaStruyve} for details). 
The even stronger fact that $G=G^{++} $ is shown in both cases, in the second paper, assuming that the local action groups are transitive and generated by point stabilisers. 
This stronger conclusion requires an ad-hoc argument that is beyond the scope of those presented here.

K.~Tent informs us that, using the techniques of \cite{Tent_simple}, it can be shown that the automorphism of a right-angled building where all panels have countably infinite thickness is boundedly simple (every element is a product of a bounded number of elements from a conjugacy class), but these techniques seem incompatible with the ones in the other papers.

\section{Hyperbolic actions and number of ends}

There is a notable difference in the large-scale geometry of the piecewise full groups appearing in Section~\ref{sec:pw_full}, as opposed to the groups acting on trees in Section~\ref{sec:tree}.  While the latter clearly act faithfully on a hyperbolic space (namely, the tree) with translations, we find that in the case of piecewise full groups, there are significant obstructions to having interesting actions on hyperbolic spaces.  In this section we adapt arguments from \cite{BFG}.

\begin{defn}\label{def:hyperbolic}
	A \defbold{hyperbolic action} of a (topological) group $G$ is a (continuous) action of $G$ by isometries on a hyperbolic space $X$.  Given such an action, then $G$ also acts on the Gromov boundary $\partial X$.  An element $g \in G$ is \defbold{loxodromic} on $X$ if $n \mapsto g^nx$ is a quasi-isometric embedding for some (hence any) $x \in X$.  A loxodromic element has exactly two fixed points $\xi_+(g),\xi_-(g)$ in $\partial X$, the \defbold{endpoints} of $g$.  We say the action of $G$ is of \defbold{general type} if there are loxodromic elements $g,h \in G$ such that $\{\xi_+(g),\xi_-(g)\}$ is disjoint from $\{\xi_+(h),\xi_-(h)\}$.
	
	A group $G$ has \defbold{property $(\mathrm{NGT})$} if there does not exist a hyperbolic action of $G$ of general type.
	
\end{defn}

In \cite{BFG}, the following theorem is proved.

\begin{thm}[{\cite[Theorem~5.1]{BFG}}]\label{thm:BFG}
Let $G$ be a discrete group acting on a compact Hausdorff space $X$.  Let $\mc{I}$ be a base of topology consisting of nondense and non-empty open subsets of $X$.  For $n \ge 1$, let $\mc{I}^{(n)}$ denote the set of $n$-tuples $(I_1,\dots,I_n)$ of elements of $\mc{I}$ such that $\bigcup^n_{i=1}I_i$ is not dense and $I_i$ and $I_j$ have disjoint closures for $i \neq j$.  Suppose that the following holds:
\begin{enumerate}[1.]
		\item $(\mathbf{C})$: For all $I \in \mc{I}$ there is $J \in \mc{I}$ such that $I^c \subseteq J$.
		\item $(\mathbf{2T})$: $G$ acts transitively on $\mc{I}^{(2)}$.
		\item $(\mathbf{3T})$: For all $g,h \in G$, there exists $(M,N,P) \in \mc{I}^{(3)}$ such that $(M,N,P)$ and $(gM,hM,P)$ are in the same $G$-orbit on $\mc{I}^{(3)}$.
		\item $(\mathbf{L})$: Let $(I,J,K) \in \mc{I}^{(3)}$ and let $g,h \in G$ such that $gI = I$ and $hJ = J$.  Then there exists $b \in G$ such that $b|_I = g|_I$, $b|_J = h|_J$ and $b|_K = \mathrm{id}|_K$.
\end{enumerate}
Then $G$ has property $(\mathrm{NGT})$.
\end{thm}

Now let $\mc{C} = \{\rist_G(I^c) \mid I \in \mc{I}\}$.  We make the following observations on the hypotheses of Theorem~\ref{thm:BFG}.
\begin{enumerate}[1.]
\item The condition $(\mathbf{2T})$ implies that $G$ acts transitively on $\mc{I}$, so $\mc{C}$ is a single conjugacy class of subgroups of $G$.  Condition $(\mathbf{L})$ ensures that some $A \in \mc{C}$, hence every $A \in \mc{C}$, is non-trivial.
\item Combining conditions $(\mathbf{C})$ and $(\mathbf{2T})$, it is easy to see that the action is fully compressible and every non-empty open set contains $I^c$ for some $I \in \mc{I}$.  We are thus in the situation of Corollary~\ref{cor:micro-supported_compression}, so $\mc{C}$ is a joinable compression family of subgroups of $G$.
\item For the proof of Theorem~\ref{thm:BFG}, it is sufficient to replace $\mc{I}^{(3)}$ with a set $\mc{I}_3$ of triples in $\mc{I}$ such that the analogues of $(\mathbf{3T})$ and $(\mathbf{L})$ hold, and for there to exist a $G$-orbit $\mc{I}_2$ on pairs in $\mc{I}$ with suitable properties to substitute for $(\mathbf{2T})$.  Since the argument is based on pointwise fixators, there is then no loss in rephrasing these properties in terms of the sets
\[
\mc{C}_n := \{(\rist_G(I^c_1),\dots, \rist_G(I^c_n)) \mid (I_1,\dots,I_n) \in \mc{I}_n\} \quad (n=2,3).
\]
The details will become clear in the statement and proof of the next theorem.
\item For the proof it is necessary to choose a certain subset $\mc{A}$ of $\mc{I}$ forming a cover of $X$, chosen so that for every $I,J \in \mc{A}$ there exists $K \in \mc{A}$ such that $(I,K),(J,K) \in \mc{I}^{(2)}$.  We will impose analogous conditions on $\mc{S} = \{\rist_G(I^c) \mid I \in \mc{A}\}$ and $\mc{C}_2$, noting that since $\mc{A}$ is a cover, the set $\mc{S}$ will have trivial intersection.
\end{enumerate}

We are now ready to state a version of \cite[Theorem~5.1]{BFG} for joinable compression families.

\begin{thm}\label{thm:NGT}
	Let $G$ be a group and let $\mc{C}$ be a conjugacy class of subgroups of $G$; fix a $G$-orbit $\mc{C}_2$ on pairs in $\mc{C}$, a finite subset $\mc{S}$ of $\mc{C}$, and a subgroup $\Der(G) \le G_0 \le G$.  Write $\mc{C}_3$ for the set of triples $(A_1,A_2,A_3)$ of elements of $\mc{C}$ with the following property:
	\begin{enumerate}
	\item[$(\mathbf{L})$] The intersection $\bigcap^3_{i=1} A_i$ contains an element of $\mc{C}$, and given $g_i \in \N_G(A_i)$ such that $g_j \in A_j$ for at least one $j \in \{1,2,3\}$, then $\bigcap^3_{i=1}g_iA_i \neq \emptyset$.
	\end{enumerate}
	We make the following assumptions:
	\begin{enumerate}
		\item[$(\mathbf{J})$] The set $\mc{C}$ is a joinable compression family for $G$.
		\item[$(\mathbf{2T})$] Given $(A,B) \in \mc{C}_2$, there exists $(A,B,C) \in \mc{C}_3$; and given $(A,B) \in \mc{C}_2$ and $A' \in \mc{C}$ such that $A' \ge A$, then $(A',B) \in \mc{C}_2$.
		\item[$(\mathbf{3T})$] For all $g,h \in G$, there exists $(A,B,C) \in \mc{C}_3$ such that $(A,B,C)$ and $(gAg\inv,hBh\inv,C)$ are in the same $G$-orbit on $\mc{C}_3$.
		\item[$(\mathbf{S})$] The set $\mc{S}$ has trivial intersection, and for every $A,B \in \mc{S}$ there is some $C \in \mc{S}$ such that $(A,C),(B,C) \in \mc{C}_2$.
	\end{enumerate}
	Then $G_0$ has property $(\mathrm{NGT})$.
	
\end{thm}

\begin{proof}
We follow the proof strategy of \cite[Theorem~5.1]{BFG}, with some extra work to account for subgroups containing $\Der(G)$. Write $S$ for the union of all elements of $\mc{S}$, and write $T$ for the union of all elements of $\mc{C}$.  

Let $G_0$ be a subgroup of $G$ such that $\Der(G) \le G_0 \le G$.  We will not need to consider isometric actions of $G_0$ directly, as we can appeal to theorems of Genevois.  By \cite[Theorem~1.1]{Genevois}, to prove that $G_0$ has $(\mathrm{NGT})$, it is enough to prove the following properties:
	\begin{enumerate}[(1)]
		\item Every element of $G_0$ can be written as the product of $3$ elements of $S \cap G_0$.
		\item Given $t,t' \in T \cap G_0$ there exist $g,h \in G_0$ such that
		\[
		[t,gtg\inv] = [gtg\inv,hth\inv] = [hth\inv,t']=1.
		\]
		\item There exists $r \ge 1$ such that for every $g,h \in G_0$ there exist $t,t_1,\dots,t_r \in T \cap G_0$ such that for every $s \in S \cap G_0$ there exists $f \in \grp{t_1} \dots \grp{t_r}$ such that each of $fsf\inv t,fsf\inv tg,fsf\inv th$ belongs to $T \cap G_0$.
	\end{enumerate}
	
By $(\mathbf{2T})$ and $(\mathbf{3T})$, we observe that for all $(A,B) \in \mc{C}_2$ the intersection $A \cap B$ contains an element of $\mc{C}$, hence a $G$-conjugate of any given element of $\mc{C}$.  Note also that $\mc{C}_3$ is stable under permuting the entries of triples.  

	We now proceed by a series of claims.

\emph{Claim 1: Given $(A,C),(B,C) \in \mc{C}_2$, there is $c \in C$ and $g \in \Der(G)$ such that $[c,g] \in C$ and $[c,g]A[g,c] = B$.}

Since $G$ is transitive on $\mc{C}_2$, there is $h \in G$ such that $hAh\inv = B$ and $hCh\inv = C$.  By $(\mathbf{2T})$ we can take $(A,C,D) \in \mc{C}_3$, and so by $(\mathbf{L})$ there is $k \in A \cap gC$.  Then $hk\inv A kh\inv = B$ and $c:=hk\inv \in C$.  We now take $g \in \Der(G)$ such that $gCg\inv \le A \cap C$, which exists by Proposition~\ref{prop:compressible_splittable}.  Then $[c,g] \in C$ and $[c,g]A[g,c] = B$.

\emph{Claim 2: Let $\Der(G) \le G_0 \le G$ and $g \in G_0$.  Then $g = s_1s_2s_3$ for some $s_1,s_2,s_3 \in S \cap G_0$.}
	
	Let $g \in G_0$ and $A \in \mc{S}$.  Then $gAg\inv \in \mc{C}$, and by Lemma~\ref{lem:leafless}(ii), some $A_2 \in \mc{C}$ commutes with $gAg\inv$.  Since $\mc{S}$ has trivial intersection, there is some $B \in \mc{S}$ such that $A_2 \nleq B$.  By joinability there is $A' \in \mc{C}$ such that $gA'g\inv \ge \grp{gAg\inv,B}$.
	
	Let $C \in \mc{S}$ such that $(A,C),(B,C) \in \mc{C}_2$, which exists by $(\mathbf{S})$.  Since $A' \ge A$, by $(\mathbf{2T})$ we have $(A',C) \in \mc{C}_2$; thus by Claim 1 there exists $c_1 \in C \cap \Der(G)$ such that $c_1Bc\inv_1 = A'$.
	
	We now have $gc_1Bc\inv_1 g\inv = gA'g\inv$.  By Claim 1 there exists $c_2 \in C \cap \Der(G)$ such that $c_2(gA'g\inv)c\inv_2 = B$.  Now we have $(c_2gc_1)B(c_2gc_1)\inv = B$.  By $(\mathbf{2T})$ we can extend $(B,C)$ to some $(B,C,D) \in \mc{C}_3$, and then by $(\mathbf{L})$, we obtain $c_3 \in (c_2gc_1)B \cap C$.  Take $k \in G$ such that $kCk\inv \le B \cap C$ and let $c_4 = [c_3,k]$.  Then $c\inv_4 c_2 \in C \cap G_0$; writing $s_2 = c\inv_4c_2gc_1$, we also have $s_2 \in B$ (since $c_4 \in (c_2gc_1)B$) and $s_2 \in G_0$ (since $s_2$ is expressed as a product of elements of $G_0$).  We now set
	\[
	s_1 = c\inv_2 c_4, \; s_2 = c\inv_4c_2gc_1, \; s_3 = c_1\inv,
	\]
	and we see that $G$ is expressed as a product of $3$ elements of $S \cap G_0$.
\

By Claim 2, we see that property (1) is satisfied.  By Theorem~\ref{thm:compressible_splittable_union} we then immediately see that $\Der(G)$ is simple, and by Proposition~\ref{prop:compressible_splittable} the action of $\Der(G)$ on $\mc{C}$ is transitive.  In particular, we deduce property (2) from Lemma~\ref{lem:leafless}(iv).  It now remains to check property (3).
	
	\emph{Claim 3: Property (3) is satisfied for $r = 3|\mc{S}|$.}
	
	Fix $g,h \in G_0$.  By $(\mathbf{3T})$ there exist $(A,B,C) \in \mc{C}_3$ and $b_0 \in \N_G(C)$ such that $b_0g \in \N_G(A)$ and $b_0h \in \N_G(B)$.  Now we use $(\mathbf{L})$ to obtain
	\[
	c_0 \in b_0gA \cap B \cap b_0C \text{ and } c_1 \in A \cap b_0hB \cap C.
	\]
	Setting $c_2 := c\inv_1 c\inv_0 b_0$, it follows that
	\[
	c_2g \in A, \; c_2h \in B, \; c_2 \in C.
	\]
	We then choose $k \in G$ such that $kCk\inv \le A \cap B \cap C$, and set $c = [k,c\inv_2]$.  Then
	\[
	cg \in A \cap G_0, \; ch \in B \cap G_0, \; c \in C \cap G_0.
	\]
	
	Given $s \in S \cap G_0$, take $D \in \mc{S}$ such that $s \in D$.  Since $G_0$ acts transitively on $\mc{C}$ there is $f = f^{(D)} \in G_0$, depending only on $D$, such that $fDf\inv \le A \cap B \cap C$.  Then
	\[
	fsf\inv cg \in A \cap G_0, \; fsf\inv ch \in B \cap G_0, \; fsf\inv c \in C \cap G_0;
	\]
	in particular, we have $fsf\inv cg, fsf\inv ch, fsf\inv c \in T \cap G_0$.
	
	By Claim 1, there exist $s^{(D)}_1, s^{(D)}_2, s^{(D)}_3 \in S \cap G_0$ such that $f^{(D)} = s^{(D)}_1s^{(D)}_2s^{(D)}_3$; in particular, $s^{(D)}_i \in T \cap G_0$ for all $D \in \mc{S}$ and $i \in \{1,2,3\}$.  Put some order on $\mc{S}$ and let $R = \prod_{D \in \mc{S}}\prod^3_{i=1}\grp{s^{(D)}_i}$.  We see that for all $s \in S \cap G_0$ we can find $f \in R$ satisfying property (3), and hence $r = 3|\mc{S}|$ suffices for property (3), for all $g,h \in G_0$.
	
	\
	
	Now that all the properties have been shown, the fact that $G_0$ has $(\mathrm{NGT})$ follows from \cite[Theorem~1.1]{Genevois}.
\end{proof}
	
Under the hypotheses of Theorem~\ref{thm:NGT}, the analogue of Corollary~\ref{cor:alternatable_NGT:ends} is now straightforward.

\begin{cor}
Let $G$ and $G_0$ be as in Theorem~\ref{thm:NGT}, such that $G_0$ is equipped with a locally compact group topology.  Suppose $G_0$ is compactly generated with no non-trivial continuous homomorphism to $\Rb$.  Then $G_0$ has at most one end.
\end{cor}

\begin{proof}
Suppose $G_0$ has more than one end.  Then $G_0/\ol{\Der(G)}$ is an abelian locally compact group that has no non-trivial continuous homomorphism to $\Rb$, from which we see that $G_0/\ol{\Der(G)}$ has at most one end.  Thus $\ol{\Der(G)}$ is not compact, and hence $G_0$ has no non-trivial compact normal subgroup; we note also that $G_0$  has no non-trivial abelian normal subgroup.  By standard results on number of ends, for instance by \cite[Theorems~3.7, 4.1, 4.2]{Houghton}, we see that $G_0$ has infinitely many ends and is totally disconnected.

By Theorem~\ref{thm:NGT}, $G_0$ has no general type actions on trees.  Since $G_0$ is also totally disconnected with infinitely many ends, it follows from \cite[Proposition~3.6]{CMR} that $G_0$ has a focal action on a locally finite tree; in particular, $G_0$ admits $\Zb$ as a continuous quotient.  This contradicts our hypothesis and we conclude that in fact $G_0$ has at most one end.
\end{proof}

We now return to the context of piecewise full groups, and find that many such groups and their derived groups have property $(\mathrm{NGT})$.  In particular, this will prove Theorem~\ref{thm:alternatable_NGT} and Corollary~\ref{cor:alternatable_NGT:ends}.

\begin{cor}\label{cor:alternatable_NGT}
Let $X$ be a perfect zero-dimensional compact Hausdorff space and let $G$ be a compressible piecewise full group of homeomorphisms of $X$, and let $\Der(G) \le G_0 \le G$.  Let $U$ be a proper non-empty clopen set, and let $\mc{C} = \{\rist_G(gU^c) \mid g \in G\}$.  Then the conditions of Theorem~\ref{thm:NGT} are satisfied, so $G_0$ has property $(\mathrm{NGT})$.
\end{cor}

\begin{proof}
By Lemma~\ref{lem:pw_full} the action of $G$ on $X$ is fully compressible and micro-supported, from which it follows via Corollary~\ref{cor:micro-supported_compression} that $\mc{C}$ is a joinable compression family for $G$, so $(\mathbf{J})$ holds.  Write $\mc{I} = \{gU \mid g \in G\}$.

We define $\mc{C}_3$ as in Theorem~\ref{thm:NGT}.  Using the fact that $G$ is piecewise full, it is straightforward to see that $\mc{C}_3$ contains the set $\mc{C}'_3$ of all triples $(\rist_G(I),\rist_G(J),\rist_G(K))$ with $I,J,K \in \mc{I}$ such that the closures of $I,J,K$ are pairwise disjoint and such that $I \cup J \cup K$ is not dense, and moreover that $G$ acts transitively on $\mc{C}'_3$.  We take $\mc{C}_2$ to be the set of all pairs $(A,B)$ such that $\mc{C}_3$ has an element of the form $(A,B,C)$.  Properties $(\mathbf{2T})$ and $(\mathbf{3T})$ are now clear.

Finally, we must obtain a set $\mc{S}$ as in Theorem~\ref{thm:NGT}; some care is needed here, as it is not clear whether $\mc{I}$ is a base of topology.

We first obtain a subset $\mc{I}_0 \subseteq \mc{I}$ of $X$ with $|\mc{I}_0| = 3$ that is minimal among covers of $X$.  Since the action is fully compressible, there are $I,J \in \mc{I}$ such that $\{I,J\}$ is a cover of $X$.  Take $K_1,K_2,K_3,L \in \mc{I}$ pairwise disjoint such that $K_1$ and $K_2$ are disjoint from $I$ and such that $L$ and $K_3$ are disjoint from $J$.  Then there is some $g \in G$ such that $gK_i = K_{i-1}$ (taking the subscript modulo $3$) and $g$ fixes all points outside of $K_1 \cup K_2 \cup K_3$.  Now write $J_1 = I, J_2 = gJ, J_3 = g^2J$ and let $\mc{I}_0 = \{J_1,J_2,J_3\}$.  We see that $J \subseteq gJ \cup g^2J$, so $\mc{I}_0$ covers $X$.  We see also that each of $J_1,J_2,J_3$ is essential in any subcover of $\mc{I}_0$: specifically, $I$ is the only element to intersect $L$, $gJ$ is the only element to intersect $K_1$ and $g^2J$ is the only element to intersect $K_2$.  Thus $\mc{I}_0$ is a minimal cover with $3$ elements.

By the minimality of $\mc{I}_0$, and by the fact that $X$ is a normal space, there is some non-empty open subset $O_i$ of $J_i$ such that $\ol{O_i}$ is disjoint from the closure of the union of $\mc{I}_0 \setminus \{J_i\}$.  Since the action is fully compressible, there exists $U_i \in \mc{I}$ such that $U_i \subseteq O_i$.  Now set $\mc{J} = \mc{I}_0 \cup \{U_1,U_2,U_3\}$.  Given $Y,Z \in \mc{J}$, we see that each of $\ol{Y}$ and $\ol{Z}$ intersects at most one of $\ol{U_1},\ol{U_2},\ol{U_3}$, so there is some $i \in \{1,2,3\}$ such that $\ol{U_i}$ is disjoint from both $\ol{Y}$ and $\ol{Z}$.

We now take $\mc{S} = \{\rist_G(A^c) \mid A \in \mc{J}\}$; given the form taken by $\mc{C}_2$, it is now easy to see that $(\mathbf{S})$ is satisfied.

All hypotheses of Theorem~\ref{thm:NGT} are now satisfied, and the conclusion follows.
\end{proof}

\section{Compressible actions of locally compact groups}

\subsection{Normal subgroups containing contraction groups}\label{sec:open_monolith}

We now prove Corollary~\ref{cor:tree:simpleG++} and Theorem~\ref{thm:loc_decomp_compressible}, first appealing to a fact about contraction groups in \tdlc groups.

\begin{defn}
	Let $G$ be a topological group and let $g \in G$.  The \defbold{contraction group} of $g$ in $G$ is
	\[
	\con_G(g) := \{u \in G \mid g^nug^{-n} \rightarrow 1 \text{ as } n \rightarrow \infty\}.
	\]
\end{defn}

\begin{lem}[{See \cite[Proposition~5.1]{CRW-TitsCore}}]\label{lem:contraction}
	Let $G$ be a \tdlc group and let $A$ be a normal subgroup of $G$ (not necessarily closed).  Then $\ol{\con_G(a)} \le A$ for all $a \in A$.
\end{lem}
	
\begin{proof}[Proof of Corollary~\ref{cor:tree:simpleG++}]
	Note that $G$ is totally disconnected, since $\Aut(T)$ is totally disconnected.  We may assume that $G^{++}$ is non-trivial, ensuring that $T$ is of general type; since $T$ admits a geometrically dense action, in fact $T$ is leafless and without isolated ends.  Let $H$ be the rigid stabiliser of some half-tree $T_a$. 
	By Corollary~\ref{cor:tree}, the action of $\Der(G^{++})$ is geometrically dense, so by Lemma~\ref{lem:tree:geom_dense}, there is a translation $g \in \Der(G^{++})$ whose axis contains the edge defining the half-tree $T_a$ and sending $T_a$ to a half-tree properly contained in $T_a$, so
	 $\bigcap_{n \ge 0} g^n(T_a) = \emptyset$.
	Since $G$ carries the permutation topology,  $H \le \con_G(g)$.  Hence $H \le \Der(G^{++})$ by Lemma~\ref{lem:contraction}.  Since the half-tree $T_a$ was arbitrary we conclude that $G^{++} \le \Der(G^{++})$, that is, $G^{++}$ is perfect.  We deduce from Corollary~\ref{cor:tree} that in fact $G^{++}$ is simple.
\end{proof}

\begin{rmk}
	It is likely that analogous arguments can be used to show that a locally compact automorphism group $G$ of a CAT(0) cube complex such as the ones considered in \cite{HaglundPaulin} and \cite{Lazarovich} has simple $G^{++}$. Checking the details of this is beyond the scope of the present paper.
\end{rmk}

\begin{proof}[Proof of Theorem~\ref{thm:loc_decomp_compressible}]
	Let $Z$ be a non-empty compressible clopen set for the action of $A$.  We see that the space $\bigcup_{g \in A}gZ$ is open and $G$-invariant, hence equal to $X$, so $X$ is covered by $A$-translates of $Z$.  Since $Z$ is compressible for $A$, it follows that $X$ is covered by $A$-translates of any non-empty open subset of $X$.  Thus the action of $A$ on $X$ is minimal.
		
	If $G$ is discrete the remaining conclusions are clear, so let us assume $G$ is not discrete.  We claim that for every non-empty clopen $Y \subseteq X$, the rigid stabiliser $\rist_G(Y)$ is non-discrete.  Since $G$ acts minimally and $X$ is compact, there exist $g_1,\dots,g_m \in G$ such that $X = \bigcup^m_{i=1}g_iY$.  There is then a partition $\mc{P} = \{Y_1,\dots,Y_m\}$ of $X$ such that $Y_i \subseteq g_iY$, and by local decomposability, $G$ has an open subgroup $\prod^m_{i=1}\rist_G(Y_i)$.  It follows that $\rist_G(Y_i)$ is non-discrete for some $i$; but then
	\[
	\rist_G(Y_i) \le \rist_G(g_iY) = g_i\rist_G(Y)g\inv_i,
	\]
so $\rist_G(Y)$ is non-discrete.
	
	Let $K = \rist_G(Z)$.  By compressibility there is $g \in A$ such that $gZ$ is properly contained in $Z$, and it follows that $K$ has an open subgroup of the form $gKg\inv \times L$ where $L = \rist_G(Z \setminus gZ)$.  By Lemma~\ref{lem:contraction}, we have $\con_G(g) \le A$.  Moreover, since the action is locally weakly decomposable in the terminology of \cite{CRW-Part2}, we see by \cite[Proposition~6.14]{CRW-Part2} that the intersection $L^* = \con_G(g) \cap L$ is open in $L$.  Since $X$ is compact and the action of $A$ on $X$ is minimal, there are $a_1,\dots,a_n \in A$ such that $X = \bigcup^n_{i=1}a_i(Z \setminus gZ)$.  By local decomposability it follows that $\grp{a_iL^*a\inv_i \mid 1 \le i \le n}$ is open in $G$; thus $A$ is open in $G$.  In particular, $A$ inherits from $G$ the property of being locally decomposable.
\end{proof}

\subsection{Robustly monolithic groups and compressibility}\label{sec:robustly_monolithic}

We recall some known properties of groups of groups in the class $\ms{R}$.

\begin{lem}\label{lem:simple_dynamics}
	Let $G \in \ms{R}$.
	\begin{enumerate}[(i)]
	\item \textup{(See \cite[Proposition~5.1.2]{CRW-DenseLC})} There are no non-trivial abelian locally normal subgroups of $G$.
	\item \textup{(See \cite[Section~5.1]{CRW-Part1})} The following is a Boolean algebra:
		\[
		\mathrm{LC}(G):= \{ \CC_G(K) \mid K \le G \text{ locally normal}\}.
		\]
	The set of elements $K$ of $\mathrm{LC}(G)$ such that $K \times \CC_G(K)$ is open in $G$ form a subalgebra $\mathrm{LD}(G)$.
	\item \textup{(See \cite[Theorem~7.3.3]{CRW-DenseLC})} Write $\Omega_G$ for the Stone space of $\mathrm{LC}(G)$, equipped with the action induced by conjugation on $\mathrm{LC}(G)$.  The action of $\TMon(G)$ on $\Omega_G$, and hence on every quotient $G$-space of $\Omega_G$, is minimal and compressible.  Consequently the action of $G$ on any non-trivial quotient $G$-space of $\Omega_G$ is faithful.
	\end{enumerate}
\end{lem}

 If $G$ has an open subgroup of the form $A \times B$ where $A$ and $B$ are non-trivial closed subgroups, then the Stone space $X = \mf{S}(\mathrm{LD}(G))$ of $\mathrm{LD}(G)$ is a non-trivial quotient $G$-space of $\Omega_G$; moreover, the action of $G$ on $X$ is locally decomposable.  Corollary~\ref{cor:R_open_monolith} now follows from Lemma~\ref{lem:simple_dynamics} and Theorem~\ref{thm:loc_decomp_compressible}.
 
It is currently unknown to the authors whether the action of $G$ on $\Omega_G$ is fully compressible for all $G \in \ms{R}$.  In order for $G \in \ms{R}$ to act fully compressibly on $\Omega_G$, it is enough for there to exist \emph{some} non-trivial $G$-space $X$ on which the action is continuous, micro-supported and fully compressible, see \cite[\S2.2]{LeBoudecURS}.

		\bibliographystyle{amsalpha} 
		\bibliography{biblio}

\def\cprime{$'$} \def\cprime{$'$}
\providecommand{\bysame}{\leavevmode\hbox to3em{\hrulefill}\thinspace}
\providecommand{\MR}{\relax\ifhmode\unskip\space\fi MR }
\providecommand{\MRhref}[2]{%
  \href{http://www.ams.org/mathscinet-getitem?mr=#1}{#2}
}
\providecommand{\href}[2]{#2}
\begin{thebibliography}{CRW17b}

\bibitem[BEW15]{BanksElderWillis}
Christopher Banks, Murray Elder, and George~A. Willis, \emph{Simple groups of
  automorphisms of trees determined by their actions on finite subtrees}, J.
  Group Theory \textbf{18} (2015), no.~2, 235--261 (English).

\bibitem[BFFG24]{BFG}
Sahana Balasubramanya, Francesco Fournier-Facio, and Anthony Genevois,
  \emph{Property ({NL}) for group actions on hyperbolic spaces (with an
  appendix by alessandro sisto}, Groups Geom. Dyn. (2024), published on-line
  first DOI 10.4171/GGD/806.

\bibitem[Cap14]{Caprace-buildings}
Pierre-Emmanuel Caprace, \emph{Automorphism groups of right-angled buildings:
  simplicity and local splittings.}, Fundam. Math. \textbf{224} (2014), no.~1,
  17--51 (English).

\bibitem[CMR24]{CMR}
Pierre-Emmanuel Caprace, Timoth\'{e}e Marquis, and Colin~D. Reid, \emph{Growing
  trees from compact subgroups}, Groups, Geometry, and Dynamics \textbf{18}
  (2024), no.~1, 327--352.

\bibitem[CRW14]{CRW-TitsCore}
Pierre-Emmanuel Caprace, Colin~D. Reid, and George~A. Willis, \emph{Limits of
  contraction groups and the {T}its core}, Journal of Lie Theory \textbf{24}
  (2014), 957--967.

\bibitem[CRW17a]{CRW-Part1}
\bysame, \emph{Locally normal subgroups of totally disconnected groups. {P}art
  {I}: {G}eneral theory}, Forum Math. Sigma \textbf{5} (2017), e11.

\bibitem[CRW17b]{CRW-Part2}
\bysame, \emph{Locally normal subgroups of totally disconnected groups. {P}art
  {I}{I}: {C}ompactly generated simple groups}, Forum Math. Sigma \textbf{5}
  (2017), e12.

\bibitem[CRW21]{CRW-DenseLC}
P.-E. Caprace, Colin~D. Reid, and Phillip~R. Wesolek, \emph{Approximating
  simple locally compact groups by their dense locally compact subgroups}, Int.
  Math. Res. Not. (2021), no.~7, 5037--5110.

\bibitem[DMSS18]{deMedtsSilvaStruyve}
Tom De~Medts, Ana~C. Silva, and Koen Struyve, \emph{Universal groups for
  right-angled buildings}, Groups Geom. Dyn. \textbf{12} (2018), no.~1,
  231--287.

\bibitem[Gen19]{Genevois}
Anthony Genevois, \emph{Hyperbolic and cubical rigidities of {T}hompson's group
  {V}}, J. Group Theory \textbf{22} (2019), no.~2, 313--345.

\bibitem[GR24]{GR-generaltheory}
A.~Garrido and C.D. Reid, \emph{Locally compact piecewise full groups of
  homeomorphisms}, in preparation, 2024.

\bibitem[Hou74]{Houghton}
C.~H. Houghton, \emph{Ends of locally compact groups and their coset spaces},
  J. Aust. Math. Soc. \textbf{17} (1974), no.~3, 274--284.

\bibitem[HP98]{HaglundPaulin}
Fr{\'e}d{\'e}ric Haglund and Fr{\'e}d{\'e}ric Paulin, \emph{Simplicity of
  automorphism groups of spaces with negative curvature}, The Epstein Birthday
  Schrift dedicated to David Epstein on the occasion of his 60th birthday,
  Warwick: University of Warwick, Institute of Mathematics, 1998, pp.~181--248
  (French).

\bibitem[Kap99]{Kapoudjian}
Christophe Kapoudjian, \emph{Simplicity of {Neretin}'s group of
  spheromorphisms.}, Ann. Inst. Fourier \textbf{49} (1999), no.~4, 1225--1240.

\bibitem[Laz18]{Lazarovich}
Nir Lazarovich, \emph{On regular {CAT}(0) cube complexes and the simplicity of
  automorphism groups of rank-one {CAT}(0) cube complexes}, Comment. Math.
  Helv. \textbf{93} (2018), no.~1, 33--54.

\bibitem[LB16]{LB-AP}
Adrien Le~Boudec, \emph{Groups acting on trees with almost prescribed local
  action}, Commentarii Mathematici Helvetici \textbf{91} (2016), no.~2,
  253--293.

\bibitem[LB21]{LeBoudecURS}
\bysame, \emph{Amenable uniformly recurrent subgroups and lattice embeddings},
  Ergodic Theory and Dynamical Systems \textbf{41} (2021), no.~5, 1464--1501.

\bibitem[Led19]{Lederle}
Waltraud Lederle, \emph{Coloured {N}eretin groups}, Groups Geom. Dyn.
  \textbf{13} (2019), no.~2, 467--510.

\bibitem[Mat15]{MatuiSimple}
Hiroki Matui, \emph{Topological full groups of one-sided shifts of finite
  type}, J. Reine Angew. Math. \textbf{705} (2015), 35--84.

\bibitem[MV12]{MollerVonk}
R.~G. M\"oller and J.~Vonk, \emph{Normal subgroups of groups acting on trees
  and automorphism groups of graphs}, J. Group Theory \textbf{15} (2012),
  no.~6, 831--850.

\bibitem[Nek19]{Nekra}
Volodymyr Nekrashevych, \emph{Simple groups of dynamical origin}, Ergodic
  Theory and Dynamical Systems \textbf{29} (2019), no.~3, 707--732.

\bibitem[RS]{ReidSmith}
Colin~D. Reid and Simon~M. Smith, \emph{Groups acting on trees with {T}its'
  independence property ({P})}, arXiv:2002.11766v2.

\bibitem[Tit70]{Tits70}
Jacques Tits, \emph{Sur le groupe des automorphismes d'un arbre}, Essays on
  topology and related topics ({M}\'emoires d\'edi\'es \`a {G}eorges de
  {R}ham), Springer, New York, 1970, pp.~188--211.

\bibitem[TZ13]{Tent_simple}
Katrin Tent and Martin Ziegler, \emph{On the isometry group of the {Urysohn}
  space}, J. Lond. Math. Soc., II. Ser. \textbf{87} (2013), no.~1, 289--303.

\end{thebibliography}

\end{document}